\documentclass[times,11pt]{article}
\usepackage{amsmath,amssymb,amsfonts,latexsym,amsthm,enumerate,url}
\usepackage{times}
\date{}

\newtheorem{theorem}{Theorem}[section]
\theoremstyle{remark}
\theoremstyle{definition}

\newtheorem{definition}[theorem]{Definition}

\newtheorem{lemma}[theorem]{Lemma}

\newtheorem{rem}[theorem]{Remark}
\newtheorem{notation}[theorem]{Notation}

\newtheorem{proposition}[theorem]{Proposition}

\newtheorem{corollary}[theorem]{Corollary}
\newcommand{\beq}[1]{\begin{equation}\label{#1}}
\newcommand{\enq}[0]{\end{equation}}

\newcommand{\remove}[1]{}
\newcommand{\enote}[1]{{\bf (Ehud:} {#1}{\bf ) }}

\newcommand{\ra}[0]{\rightarrow}

\begin{document}
  \title{A Quantitative Version of the Gibbard-Satterthwaite Theorem for
  Three Alternatives}
  \author{{Ehud Friedgut\thanks{ehudf@math.huji.ac.il,
The Hebrew University of Jerusalem.}}
  \quad{Gil Kalai \thanks{kalai@math.huji.ac.il, The Hebrew University of
Jerusalem and Yale University. Work supported by an NSF grant and
by a BSF grant.}}
  \quad{Nathan Keller \thanks{nathan.keller@weizmann.ac.il,
  Weizmann Institute of Science. Work supported by the Koshland
  center for basic research.}}
  \quad{Noam Nisan \thanks{noam.nisan@gmail.com, The Hebrew University of
Jerusalem}}}

\maketitle
\thispagestyle{empty}

\begin{abstract}

The Gibbard-Satterthwaite theorem states that every
non-dictatorial election rule among at least three alternatives
can be strategically manipulated. We prove a quantitative version
of the Gibbard-Satterthwaite theorem: a random manipulation by a
single random voter will succeed with a non-negligible probability
for any election rule among three alternatives that is far from
being a dictatorship and from having only two alternatives in its
range.

\end{abstract}


\section{Introduction}
A Social Choice Function (SCF), or an election rule, aggregates
the preferences of all members of a society towards a common
social choice. The study of SCFs dates back to the works of
Condorcet in the 18th century, and has expanded greatly in the
last decades.

One of the obviously desired properties of an SCF is
\emph{strategy-proofness}: a voter should not gain from voting
strategically, that is, from reporting false preferences instead
of his true preferences (such voting is called in the sequel
\emph{manipulation}). However, it turns out that this property
cannot be obtained by any reasonable SCF. This was shown in a
landmark theorem of Gibbard~\cite{Gib73} and
Satterthwaite~\cite{Sat75}:
\begin{theorem}[Gibbard, Satterthwaite]
Any SCF which is not a dictatorship (i.e., the choice is not made
according to the preferences of a single voter), and has at least
three alternatives in its range, can be strategically manipulated.
\end{theorem}

The Gibbard-Satterthwaite theorem implies that we cannot hope for
full truthfulness in the context of voting, since any reasonable
election rule can be manipulated. However, it still may be that
such a manipulation is possible only very rarely, and thus can be
neglected in practice.

In this paper we prove a \emph{quantitative} version of the
Gibbard-Satterthwaite theorem in the case of three alternatives,
showing that if the SCF is not very close to a dictatorship or to
having only two alternatives in its range, then even a random
manipulation by a randomly chosen voter will succeed with a
non-negligible probability. Thus, one cannot hope that
manipulations will be negligible for any reasonable election
rule.\footnote{We note that functions that are very close to being
a dictatorship may have a very small number of manipulable
profiles (see e.g.~\cite{MPS04}). However, all of the prominent
SCFs are far from being a dictatorship.}

In order to present our results we need a few standard
definitions. First we formally define an SCF and a profitable
manipulation.
\begin{definition}
An SCF on $n$ voters and $m$ alternatives is a function $F:(L_m)^n
\rightarrow \{1,2,\ldots,m\}$, where $L_m$ is the set of linear
orders on $m$ alternatives. A set of preferences given by the
voters, i.e., $(x_1,x_2,\ldots,x_n) \in (L_m)^n$, is called a
\emph{profile}. When we want to single out the $i$th voter, we
write the profile as $(x_i,x_{-i})$, where $x_{-i}$ denotes the
preferences of the other voters.

A \emph{profitable manipulation} by voter $i$ at the profile
$(x_1,\ldots,x_n)$ is a preference $x'_i \in L_m$, such that
$F(x'_i,x_{-i})$ is preferred by voter $i$ (according to his ``true''
preference order $x_i$) over $F(x_i,x_{-i})$. A
profile is called \emph{manipulable} if there exists a profitable
manipulation for some voter at that profile.
\end{definition}

Now we define the quantitative settings we consider. Throughout
the paper, we make the \emph{impartial culture
assumption}~\cite{Bla58}, meaning that the profiles are
distributed uniformly.\footnote{Note that we cannot hope for an
impossibility result for {\em every} distribution, e.g. since for
every SCF one may consider a distribution on its non-manipulable
profiles.}
Under the uniform probability measure, the distance of $F$ from a
dictatorship is simply the fraction of values that has to be
changed in order to turn $F$ into a dictatorship. Similarly, in
the case of three alternatives, the distance of $F$ from having
only two alternatives in its range is the minimal probability that
an alternative is elected.

We quantify the probability of manipulation in the following way:
\begin{definition} \label{d:mp}
The {\em manipulation power} of voter $i$ on an SCF $F$, denoted
by $M_i(F)$, is the probability that $x'_i$ is a profitable
manipulation of $F$ by voter $i$ at profile $(x_1, \ldots,x_n)$,
where $x_1, \ldots,x_n$ and $x'_i$ are chosen uniformly at random
from $L_m$.
\end{definition}

We note that our notion of \emph{manipulation power} resembles
notions of power and influence which play important roles in
voting theory and in theoretical computer science. Specifically,
our reliance on the uniform probability distribution makes our
notion analogous to the {\it Banzhaf Power Index} from voting
theory~\cite{Fis73}, and to Ben-Or and Linial's notion of {\it
influence}~\cite{BL}. While the latter two notions coincide for
monotone Boolean functions, the notion of Ben-Or and Linial deals
also with the influence of coalitions (i.e., larger sets of
voters) and with much more general protocols for aggregation.
Similarly, the notion of manipulation power of one and more voters
can be of interest in much greater generality.

Finally, we define a notion related to Generalized Social Welfare
Functions which plays a central role in our proof.
\begin{definition}\label{d:IIA}
A Generalized Social Welfare Function (GSWF) on $n$ voters and $m$
alternatives is a function $G:(L_m)^n \rightarrow
\{0,1\}^{{m}\choose{2}}$. That is, given the preference orders of
the voters, $G$ outputs the preferences of the society amongst
each pair of alternatives.

A GSWF $G$ satisfies the {\it Independence of Irrelevant
Alternatives} (IIA) condition if the preference of the society
amongst any pair of alternatives $(A,B)$ depends only on the
individual preferences between $A$ and $B$, and not on other
alternatives.
\end{definition}

Now we are ready to state our main theorem.
\begin{theorem} \label{t:mt} There exist universal constants
$C,C'>0$ such that for every $\epsilon>0$ and any $n$ the
following holds:
\begin{itemize}
\item If $F$ is an SCF on $n$ voters and three alternatives, such
that the distance of $F$ from a dictatorship and from having only
two alternatives in its range is at least $\epsilon$, then
\[
\sum_{i=1}^n M_i(F) \ge C \cdot \epsilon^6.
\]
\item If, in addition, $F$ is neutral (that is, invariant under
permutation of the alternatives), then:
\[
\sum_{i=1}^n M_i(F) \ge C' \cdot \epsilon^2.
\]
\end{itemize}
\end{theorem}

We note that the value of the constant $C$ obtained in our proof
of Theorem~\ref{t:mt} is extremely low (see
Remark~\ref{Rem:Const}), and thus the first assertion applies only
in the asymptotic setting. Unlike the value of $C$, the obtained
value of $C'$ is reasonable.

\medskip

The proof of the theorem consists of three steps:
\begin{enumerate}
\item \textbf{Reduction from low manipulation power to low
dependence on irrelevant alternatives:} We show that if
$\sum_{i=1}^n M_i(F)$ is small, then in some sense, the question
whether the output of $F$ is alternative $A$ or alternative $B$
depends only a little on alternatives other than $A$ and $B$.
Specifically, the probability of changing the outcome of $F$ from
$A$ to $B$ by altering the individual preferences between
\emph{all other alternatives} (and leaving the preferences between
$A$ and $B$ unchanged) is low. This reduction is obtained by a
directed isoperimetric inequality, which we prove using the FKG
correlation inequality~\cite{FKG71} (or, more precisely, using
Harris' inequality~\cite{Ha60}).

\item \textbf{Reduction from an SCF with low dependence on
irrelevant alternatives to a GSWF with a low paradox probability:}
We show that given an SCF $F$ on three alternatives with low
dependence on irrelevant alternatives, one can construct a GSWF
$G$ on three alternatives which satisfies the IIA condition and
has a low probability of paradox (a paradox occurs if for some
profile, the society prefers $A$ over $B$, $B$ over $C$ and $C$
over $A$). Furthermore, the distance of $G$ from dictatorship and
from always ranking one alternative at the top/bottom is roughly
the same as the distance of $F$ from dictatorship and from having
only two alternatives in its range, respectively.

\item \textbf{Applying a quantitative version of Arrow's
impossibility theorem:} We use the quantitative versions of
Arrow's theorem\footnote{Arrow's theorem and its quantitative
versions are described in Section~\ref{sec:second-reduction}.}
obtained by Kalai~\cite{Ka02} (in the neutral case),
Mossel~\cite{Mos09}, and Keller~\cite{Kel10} to show that since
$G$ has low paradox probability, it has to be close either to a
dictatorship or to always ranking one alternative at the
top/bottom. Translation of the result to $F$ yields the assertion
of the theorem. We note that the proofs of the quantitative Arrow
theorem are quite complex and use discrete Fourier analysis on the
Boolean hypercube and hypercontractive inequalities.
\end{enumerate}

For a fixed value of $\epsilon$, Theorem~\ref{t:mt} implies lower
bounds of $\Omega(1)$ and $\Omega(1/n)$ on $\sum_i M_i(F)$ and
$\max_i M_i(F)$, respectively. It is easy to see that the lower
bound on $\sum_i M_i(F)$ cannot be improved (up to the value of
$C$ and the dependence on $\epsilon$), and that the lower bound on
$\max_i M_i(F)$ cannot be improved to $\Omega(1)$. The latter
follows since for the plurality SCF on $n$ voters, only an
$O(1/\sqrt{n})$ fraction of the profiles can be manipulated at all
by any single voter, and thus $M_i(Plurality) = O(1/\sqrt{n})$
for all $i$. However, it is still possible that one can obtain a
better lower bound than $\Omega(1/n)$ on $\max_i M_i(F)$, and we
leave this as our first open problem.

Our second open problem concerns the case of more than three
alternatives, $m>3$. While some parts of our proof extend to this
case (see Section~\ref{sec:more-alternatives}), we were not able
to extend all required parts of the proof. After the preliminary
version of this paper was written, several papers tried to resolve
this case, and the most notable result is by Isaksson, Kindler,
and Mossel~\cite{IKM10}, who obtained a quantitative
Gibbard-Satterthwaite theorem for $m>3$ alternatives, under the
only additional assumption of neutrality (see Theorem~\ref{t:ikm}
below). However, the case of general SCFs on more than three
alternatives is still open, and we leave it as our second open
problem. We do conjecture that the theorem generalizes to $m>3$,
perhaps with the exact form of the bound decreasing polynomially
in $m$ (like the bound obtained by Isaksson et al. in the neutral
case).

\medskip

Our result can be viewed as part of the study of computational
complexity as a barrier against manipulation in elections. A brief
overview of the work in this direction, several follow-up results,
and a short discussion of their implications is presented in
Section~\ref{sec:related-work}. In
Sections~\ref{sec:first-reduction},~\ref{sec:second-reduction},
and~\ref{sec:quantitative-arrow} we present the three steps of our
proof. Finally, we discuss the case of more than three
alternatives in Section~\ref{sec:more-alternatives}.

\section{Related Work}
\label{sec:related-work}

Since the Gibbard-Satterthwaite theorem was presented, numerous
works studied ways to overcome the strategic voting obstacle. The
two best-known ways are allowing payments (see,
e.g.,~\cite{Gro73}) and restricting the voters' preferences
(see~\cite{Mou80}).

Another way, suggested in 1989 by Bartholdi, Tovey, and
Trick~\cite{BTT89}, is to use a \emph{computational barrier}. That
is, to show that there exist reasonable SCFs for which, while a
manipulation does exist, it cannot be found efficiently, and thus
in practice, the SCF can be considered strategy-proof. Bartholdi
et al.~\cite{BTT89} constructed a concrete SCF for which they
proved that finding a profitable manipulation is $NP$-hard as an
algorithmic problem. This approach was further explored by
Bartholdi and Orlin~\cite{BO91} who proved that manipulation is
$NP$-hard also for the well-known Single Transferable Vote (STV)
election rule. In a related line of research, several papers
showed that for various SCFs, the problem of \emph{coalitional
manipulation} (i.e., when a coalition of voters tries to
coordinate their ballots in order to get their favorite
alternative elected), is $NP$-hard for some SCFs, even for a
constant number of alternatives (see
~\cite{CS03,CSL07,EL05,FHH06,HHR07}).

However, while the results in this direction are encouraging, the
computational barrier they suggest against manipulation may be
practically insufficient. This is because all the results study
the \emph{worst case complexity} of manipulation, and show that
manipulation is computationally hard for specific instances. In
order to practically prevent manipulation, one should show that it
is computationally hard for most instances, or at least in the
\emph{average case}.

In the last few years, several papers considered the hardness of
manipulation in the average
case~\cite{CS06,PR06,PR07,XC08b,ZPR09}.\footnote{It should be
noted that the success probability of a random manipulation for
SCFs with a small number of voters and alternatives was studied a
long time ago in a paper of Kelly~\cite{Ke93}.} Their results suggest
that unlike worst-case complexity, it appears that various SCFs
can be manipulated relatively easily in the average case -- that
is, for an instance chosen at random according to some typical
distribution. However, all these results consider specific families
of SCFs, and manipulation by coalitions rather than by individual voters.

Our results also study manipulation in the ``average case'' by
examining the success probability of a random manipulation by a
randomly chosen voter, and yield a general impossibility result in
the case of three alternatives, showing that for any reasonable
SCF, such manipulation succeeds with a non-negligible probability.
However, our result does not have direct implications on the study
of computational hardness of manipulation, since in the case of a
constant number of alternatives, the number of possible
manipulations by a single voter is constant, and thus manipulation
by a single voter cannot be computationally hard in this setting.

\bigskip

\noindent \textbf{Follow-Up Work}
\label{sec:sub:follow-up}

\bigskip

Since the preliminary version of this paper~\cite{FKN08} appeared
in FOCS'08, three follow-up works generalized its results to more
than three alternatives, under various additional constraints.

The first follow-up work is by Xia and Conitzer~\cite{XC08}, who
use similar techniques to show that a random manipulation will
succeed with probability of $\Omega(1/n)$ for any SCF on a
constant number of alternatives satisfying the following five
conditions:
\begin{enumerate}
\item Homogeneity -- The outcome of the election does not change
if each vote is replaced by $k$ copies of it.

\item Anonymity -- The SCF treats all the voters equally.

\item Non-Imposition -- Any alternative can be elected.

\item Cancelling out -- The outcome is not changed by adding the
set of all possible linear orders of the alternatives as
additional votes.

\item A complex stability condition (see~\cite{XC08} for the exact
formulation).
\end{enumerate}

While the conditions look a bit restrictive, Xia and
Conitzer show that they hold for several well-known SCFs, including
all positional scoring rules, STV, Copeland, Maximin, and Ranked
Pairs.

\medskip

The second follow-up work is by Dobzinski and
Procaccia~\cite{DP09}. They consider the case of two voters and an
arbitrary number $m$ of alternatives, and show that if an SCF is
$\epsilon$-far from a dictatorship and satisfies Pareto optimality
(i.e., if both voters prefer alternative $A$ over $B$, then $B$ is
not elected), then a random manipulation will succeed with
probability at least $\epsilon/m^8$. The techniques used in the
proof of~\cite{DP09} are relatively simple, and the authors
suggest that possibly their result can be generalized to any number of
voters, by modifying an inductive argument of
Svensson~\cite{Sve99} that extends the proof of the classical
Gibbard-Satterthwaite theorem from two voters to $n$ voters, for
any $n$.

\medskip

The most recent, and most notable, work is by Isaksson, Kindler
and Mossel~\cite{IKM10} who prove a quantitative version of the
Gibbard-Satterthwaite theorem for a general number of
alternatives, under the only additional assumption of neutrality.
\begin{theorem}[Isaksson, Kindler, and Mossel] \label{t:ikm}
Let $F$ be a neutral SCF on $m \geq 4$ alternatives which is at
least $\epsilon$-far from a dictatorship. Consider a random
manipulation generated by choosing a profile and a manipulating
voter at random, and replacing four adjacent alternatives in the
preference order of that voter by a random permutation of them.
Then
\[
\Pr[\mbox{ The manipulation is successful }] \geq
\frac{\epsilon^2}{10^9 n^4 m^{34}}.
\]
\end{theorem}
The techniques used by Isaksson et al. are combinatorial and
geometric, and contain a generalization of the
canonical path method which allows to prove isoperimetric
inequalities for the interface of three bodies.

The result of Isaksson et al. shows that for any neutral SCF which
is far from a dictatorship, a random manipulation by a single
randomly chosen voter will succeed with a non-negligible
probability. Thus, a single voter with black-box access to the SCF
can find a manipulation efficiently.

However, this result still does not imply that the agenda of using
computational hardness as a barrier against manipulation is
completely hopeless, for three reasons:
\begin{enumerate}
\item The result relies on the assumption that the votes are
distributed uniformly. It is possible to argue that in real-life
situations, the distribution of the votes is far from uniform, and
thus the result does not apply.

\item The result applies only to neutral SCFs.

\item While the result implies that with a non-negligible
probability, a random manipulation by a randomly chosen voter
succeeds, it is still possible that for {\it most of the voters},
manipulation cannot be found efficiently (or even at all), while
only for a polynomially small portion of the voters a manipulation
can be found efficiently. Thus, it is possible that only a few
voters can manipulate efficiently, while most voters cannot.
\end{enumerate}

For an extensive overview of the study of computational complexity
as a barrier against manipulation, and a further discussion on the
implication of our results and the results of the follow-up works,
we refer the reader to the survey~\cite{FP10} by Faliszewski and
Procaccia.

\section{Reduction from Low Manipulation Power to Low Dependence
on Irrelevant Alternatives} \label{sec:first-reduction}

In this section we show that if $F$ is an SCF on three
alternatives such that the manipulation power\footnote{See
Definition~\ref{d:mp}.} of the voters on $F$ is small, then in
some sense, the dependence of $F$ on irrelevant alternatives is
low. We quantify this notion as follows:
\begin{definition}
Let $F$ be an SCF on three alternatives and let $a,b$ be two
alternatives. For a profile
$x \in (L_3)^n$, denote by $x^{a,b} \in \{0,1\}^n$ the vector
which represents the preferences of the voters between $a$ and
$b$, where $x_i^{a,b}=1$ if the $i$th voter prefers $a$ over $b$,
and $x_i^{a,b}=0$ otherwise. The dependence of the choice between
$a$ and $b$ on the (irrelevant) alternative $c$ is:
\[
M^{a,b}(F) = \Pr[F(x)=a, \:F(x')=b],
\]
where $x,x' \in (L_3)^n$ are chosen at random, subject to the
restriction $x^{a,b}=(x')^{a,b}$.
\end{definition}
By the definition, $M^{a,b}(F)$ measures how often a change of the
individual preferences between the alternative $c$ and the
alternatives $a,b$ leads to changing the output of $F$ from $a$ to
$b$ or vice versa. Thus, the notion measures how much the
(irrelevant) alternative $c$ affects the question of whether $a$ or
$b$ is elected.

We note that $M^{a,b}$ can be also viewed as measuring kind of a
manipulation, where all the voters together attempt to change the
output of $F$ to be $b$ rather than $a$ by re-choosing at random
all their preferences -- except for those between $a$ and $b$.
However, this definition does not require that anyone in
particular gains from changing the output.

The result we prove is the following:
\begin{lemma}\label{lemma:first-reduction}
Let $F$ be an SCF on three alternatives. Then for every pair of
alternatives $a,b$,
\[
M^{a,b}(F) \leq 6 \sum_{i=1}^n M_i(F).
\]
\end{lemma}

In order to prove Lemma~\ref{lemma:first-reduction}, we define a
certain combinatorial structure and relate it both to $M^{a,b}(F)$
and to $\sum_i M_i(F)$.

We begin with a convenient way to represent a profile $x \in
(L_3)^n$, given the individual preferences between $a$ and $b$
(denoted by $x^{a,b}$). Note that for any specific value
$z^{a,b}$ of $x^{a,b}$, there are exactly $3^n$ possible values of
$x$ that agree with it. Indeed, the agreement of $x$ with
$z^{a,b}$ fixes the preferences of all voters between $a$ and $b$
in $x$, and each voter may choose one of three locations for $c$:
above both $a$ and $b$, below both of them, or between them. Thus,
for every fixed $z^{a,b}$ we can view the set $\{x | x^{a,b}
=z^{a,b}\}$ as isomorphic to $\{0,1,2\}^n =
\{above,between,below\}^n$. We use $v=(v_1,\ldots,v_n)$ to denote
an element in this set. Thus, once $x_i^{a,b}$ is fixed, $v_i \in
\{0,1,2\}$ encodes both $x_i^{b,c}$ and $x_i^{c,a}$. For example,
$x_i^{a,b}=0$ and $v_i=0$ encodes the preference $c \succ_i b
\succ_i a$.

Next, we define two sets which are closely related to the
definition of $M^{a,b}(F)$.
\begin{definition}\label{papb}
For every value $z^{a,b}$ of the preferences between $a$ and $b$, let
\[
A(z^{a,b})= \{x | x^{a,b} =z^{a,b}, \ F(x)=a \}, \mbox{ and }
\]
\[
B(z^{a,b})= \{x | x^{a,b} =z^{a,b},\ F(x)=b \}.
\]
Both $A(z^{a,b})$ and $B(z^{a,b})$ are viewed as residing in the
space $\{0,1,2\}^n$.
\end{definition}
In terms of these definitions, we clearly have:
\begin{equation}\label{Eq:Mab}
M^{a,b}(f) = \mathbb{E}_{x \in (L_3)^n} \left[ {{|A(x^{a,b})|}
\over {3^n}} \cdot {{|B(x^{a,b})|} \over {3^n}}\right].
\end{equation}

In order to relate $M_i(F)$ to the sets $A(x^{a,b})$ and
$B(x^{a,b})$, we endow the set $\{0,1,2\}^n$ with a structure of a
directed graph, whose edges correspond to (some of) the profitable
manipulations by voter $i$. For each fixed value of $x^{a,b}$, for
each $i \in \{1,2,\ldots,n\}$, and for each $v_{-i} \in
\{0,1,2\}^{n-1}$, the graph has three directed edges in direction
$i$ between the possible values of $v_i$: $0 \rightarrow 1$, $1
\rightarrow 2$, and $0 \rightarrow 2$. The following definition
counts the directed edges going ``upward'' from a subset of
$\{0,1,2\}^n$.
\begin{definition}
Let $A \subseteq \{0,1,2\}^n$. The upper edge border of $A$ in the
$i$th direction, denoted by $\partial_i A$, is the set of directed
edges in the $i$th direction defined above whose tail is in $A$
and whose head is not in $A$. That is,
\[
\partial_i A = \{(v_{-i},v_i,v'_i) \ |\ (v_{-i},v_i) \in A,
\ (v_{-i},v'_i) \not \in A, \ v_i<v'_i \}.
\]
The upper edge border of $A$ is $\partial A = \bigcup_i
\partial_i(A)$.
\end{definition}
We now relate $M_i(F)$ to the upper edge borders in the $i$th
direction of $A(x^{a,b})$ and $B(x^{a,b})$.
\begin{lemma}
For any $1 \leq i \leq n$, we have:
\begin{equation}\label{Eq:New-1}
M_i(f) \ge \frac{1}{6} \cdot 3^{-n} \cdot \mathbb{E}_{x \in
(L_3)^n} \left[ |\partial_i A(x^{a,b})| + |\partial_i B(x^{a,b})|
\right].
\end{equation}
\end{lemma}
\begin{proof}
We compute a lower bound on $M_i(F)$ by choosing $x$ and $x'$ at
random, differing only (possibly) in the preferences of the $i$th
voter, and providing a lower bound on the probability that the
$i$th coordinate of $x'$ is a profitable manipulation of $x$. We
perform the random choice as follows: First we choose at random
$x^{a,b}_{-i} \in \{0,1\}^{n-1}$, $x^{a,b}_i \in \{0,1\}$, and
$x'^{a,b}_i \in \{0,1\}$. With probability $1/2$, we have
$x'^{a,b}_i = x^{a,b}_i$, and the rest of the analysis is
conditioned on this event indeed occurring (a conditioning that
does not affect the distribution chosen). We next choose $v_{-i}
\in \{0,1,2\}^{n-1}$, and finally we choose $v_i \in \{0,1,2\}$
and $v'_i \in \{0,1,2\}$.

We claim that if $(v_{-i},v_i,v'_i) \in \partial_i A$, then either
$x'_i$ is a manipulation of $x$ or $x_i$ is a manipulation of
$x'$.

Indeed, note that by the definition of $\partial_i A$, the
condition $(v_{-i},v_i,v'_i) \in \partial_i A$ implies that when
moving from $x_i$ to $x'_i$, voter $i$ lowered his relative
preference of $c$ without changing his ranking of the pair
$(a,b)$, and this changed the output of $F$ from $a$ to some other
result $t \in \{b,c\}$. We have two possible cases:
\begin{enumerate}
\item If, according to $x_i$, voter $i$ prefers $t$ to $a$, then
$x'_i$ is a manipulation of $x$.

\item If $x_i$ ranks $a$ above $t$, then this is definitely true
for $x'_i$ too, since when moving from $x_i$ to $x'_i$, $a$'s rank
relative to $b$ did not change, whereas it improved relative to
$c$. Thus, $x_i$ is a manipulation of $x'$.
\end{enumerate}
Thus, in both cases either $x'_i$ is a manipulation of $x$ or
$x_i$ is a manipulation of $x'$, as claimed.

The claim implies that every edge in $\partial_i A$ corresponds to
a different pair $(x,x')$ for which the $i$th coordinate of $x'$
is a profitable manipulation of $x$. Since each such edge is
chosen with probability $\frac{1}{2} \cdot 3^{-n} \cdot \frac 1
3$, the total contribution of such pairs to the lower bound on
$M_i(F)$ is $\frac 1 6 \cdot 3^{-n} \cdot \mathbb{E}_{x} [
|\partial_i A(x^{a,b})|.$ A similar contribution comes from the
case $(v_{-i},v_i,v'_i) \in \partial_i B$.
\end{proof}

Summing the two sides of Equation~(\ref{Eq:New-1}) over $i$, we
get:
\begin{equation}\label{Eq:New-2}
\sum_{i=1}^n M_i(F) \geq \frac {3^{-n}}{6} \cdot \mathbb{E}_{x \in
(L_3)^n} \left[ \left( |\partial A(x^{a,b})|+|\partial B(x^{a,b})|
\right) \right].
\end{equation}
Now we are ready to establish the relation between $\sum_i M_i(F)$
and $M^{a,b}(F)$. Recall that Equation~(\ref{Eq:Mab}) above
states:
\[
M^{a,b}(f) = \mathbb{E}_{x \in (L_3)^n} \left[ {{|A(x^{a,b})|}
\over {3^n}} \cdot {{|B(x^{a,b})|} \over {3^n}}\right].
\]
By combination of these two equations, the application of the
following proposition to the sets $A(x^{a,b})$ and $B(x^{a,b})$
yields the assertion of Lemma~\ref{lemma:first-reduction}.
\begin{proposition}
For every pair of disjoint sets $A,B \subset \{0,1,2\}^n$, we
have:
\[
|\partial(A)|+|\partial(B)| \ge 3^{-n}|A||B|.
\]
\end{proposition}

\begin{proof}
We start by ``shifting'' both $A$ and $B$ upward, using a standard
monotonization technique (see, e.g.,~\cite{Fra87}). The shifting
is performed by a process of $n$ steps. We denote $A_0=A$, and for
each $i = 1,\ldots,n$, at step $i$ we replace $A_{i-1}$ by a set
$A_i$ of the same size that is monotone in the $i$'th coordinate
(which means that if $v \in A_i$ and $v'_i \ge v_i$ then
$(v_{-i},v'_i) \in A_i$). This is done by moving every $v$ with
$v_i<2$ to have $v_i=2$ if the obtained element is not already in
$A$, and then moving every $v$ that remained with $v_i=0$ to have
$v_i=1$ if the obtained element is not already in $A$. Clearly
such steps do not change the size of the set, and thus $|A_i|=|A|$
for all $i$. As usual in such operations, it is not hard to check
that the step operation does not increase $\partial_j A$ for any
$j$, and in particular, does not destroy the monotonicity in
previous indices (see, e.g.,~\cite{Fra87} for similar arguments).
Hence, the sequence $|\partial_j(A_i)|$ is monotone decreasing in
$i$ for all $j$.

Let $A'$ and $B'$ be the sets we obtain after all $n$ steps. We
claim that $A' \setminus A$, the set of ``new'' elements added in
the monotonization process, satisfies
\begin{equation}\label{Eq:New-3}
|A' \setminus A| \le |\partial(A)|.
\end{equation}
Indeed, it is clear that every new element added in the $i$th step
of the monotonization corresponds to either one or two edges in
$\partial_i(A_{i-1})$ and these edges are disjoint. Thus, denoting
by $m_i$ the number of new elements added in the $i$th step, we
get by the monotonicity of the sequence $|\partial_j(A_i)|$, that:
\[
|A' \setminus A| \leq \sum_{i=1}^n m_i \leq \sum_{i=1}^n
|\partial_i(A_{i-1})| \leq \sum_{i=1}^n |\partial_i(A)| =
|\partial(A)|.
\]
Similarly, we have
\begin{equation}\label{Eq:New-4}
|B' \setminus B| \le |\partial(B)|.
\end{equation}
Since both $A'$ and $B'$ are monotone in the partial order of the
lattice $\{0,1,2\}^n$, they are ``positively correlated'', by
Harris' theorem~\cite{Ha60}, or by its better known
generalization, the FKG inequality~\cite{FKG71}. This means that
\[
|A' \cap B'|/3^n \ge |A'|/3^n \cdot |B'|/3^n = |A|/3^n \cdot |B|/3^n.
\]
However, by assumption $A$ and $B$ are disjoint and thus $A' \cap
B' \subseteq (A'\setminus A) \cup (B' \setminus B)$. Therefore, by
Equations~(\ref{Eq:New-3}) and~(\ref{Eq:New-4}), we have:
\[
|\partial(A)|+|\partial(B)| \geq |(A'\setminus A) \cup (B'
\setminus B)| \geq |A' \cap B'| \geq |A| \cdot |B|/3^n.
\]
This completes the proof of the proposition, and thus also of
Lemma~\ref{lemma:first-reduction}.
\end{proof}

\section{Reduction from an SCF with Low Dependence on Irrelevant
Alternatives to an Almost Transitive GSWF}
\label{sec:second-reduction}

In this section we present a reduction which allows to pass from
an SCF with low dependence on irrelevant alternatives to a GSWF
to which one can apply a quantitative version of Arrow's
impossibility theorem. In order to present the results, we need a
few more definitions related to GSWFs and to the quantitative
Arrow theorem.

Recall that a GSWF on $m$ alternatives is a function $G:(L_m)^n
\rightarrow \{0,1\}^{{m}\choose{2}}$ which is given the preference
orders of the voters, and outputs the preference of the society
amongst each pair $(a,b)$ of alternatives. The output preference
of $G$ between $a$ and $b$ for a given profile $x \in (L_m)^n$ is
denoted by $G^{a,b}(x) \in \{0,1\}$, where $G^{a,b}(x)=1$ if $a$
is preferred over $b$, and $G^{a,b}(x)=0$ is $b$ is preferred over
$a$. $G$ satisfies the IIA condition if for any pair $(a,b)$, the
function $G^{a,b}(x)$ depends only on the vector $x^{a,b} \in
\{0,1\}^n$ (which represents the preferences of the voters between
$a$ and $b$), and not on other alternatives.

As was shown by Condorcet in 1785, a GSWF based on the majority
rule amongst pairs of alternatives can result in a non-transitive
outcome, that is, a situation in which there exist alternatives
$a,b,c$, such that $a$ is preferred by the society over $b$, $b$
is preferred over $c$, and $c$ is preferred over $a$. The seminal
impossibility theorem of Arrow~\cite{Arr51,Arr63} asserts that
such non-transitivity occurs in any ``non-trivial'' GSWF on at
least three alternatives satisfying the IIA condition:
\begin{theorem}[Arrow]
Consider a GSWF $G$ with at least three alternatives. If the
following conditions are satisfied:
\begin{itemize}
\item The IIA condition,

\item Unanimity --- if all the members of the society prefer some
alternative $a$ over another alternative $b$, then $a$ is
preferred over $b$ in the outcome of $F$,

\item $F$ is not a dictatorship (that is, the preference of the
society is not determined by a single member),
\end{itemize}
then there exists a profile for which the outcome is
non-transitive.
\end{theorem}

Since we would like to use quantitative versions of Arrow's
theorem on three alternatives, we use the following notation:
\begin{notation}
For a GSWF $G$ on three alternatives, let
\[
NT(G) = \Pr_{x \in (L_3)^{n}} [G(x) \mbox{ is non-transitive }].
\]
\end{notation}

The family of all GSWFs on three alternatives satisfying the IIA
condition whose output is always transitive (i.e., those trivial
GSWFs for which the conclusion of Arrow's theorem does not apply)
was partially characterized by Wilson~\cite{Wil72}, and fully
characterized by Mossel~\cite{Mos09}. It consists exactly of all
the dictatorships and the anti-dictatorships (i.e., GSWFs whose
output is either the preference order of a single voter or its
inverse), and the GSWFs which rank a fixed alternative always at
the top (or always at the bottom). (See Theorem~\ref{Thm:Mossel3}
for the exact formulation.) Clearly, all such GSWFs are
undesirable from the point of view of Social choice theory, and
one may assume that a reasonable GSWF is ``far'' from being
contained in this set. To quantify this notion, we denote
\[
TR_3=\{\mbox{ All GSWFs on 3 alternatives which satisfy IIA and
are always transitive }\},
\]
and for a GSWF $G$ on three alternatives, denote by
\[
Dist(G,TR_3)
\]
the minimal fraction of output values that should be changed in
order to make $G$ always transitive, while maintaining the IIA
condition. The quantitative versions of Arrow's theorem which we
use in the next section assert that if $NT(G)$ is small (i.e., $G$
is almost transitive), then $Dist(G,TR_3)$ must be small as well
(and thus, $G$ is close to the family of ``bad'' GSWFs).

Another definition that will be used in the proof is the following:
\begin{definition}
For a GSWF $G$ on $m$ alternatives, and a profile $x \in (L_m)^n$,
we say that an alternative $a$ is a Generalized Condorcet Winner
(GCW) at profile $x$ if for any alternative $b \neq a$, we have
$G^{a,b}(x)=1$. A Generalized Condorcet Loser (GCL) is defined
similarly.
\end{definition}

Now we are ready to present our result.
\begin{lemma}\label{lemma:second-reduction}
Let $\epsilon_1,\epsilon_2>0$, and let $F$ be an SCF on three
alternatives, such that:
\begin{itemize}
\item $M^{a,b}(F) \leq \epsilon_1$ for all pairs $(a,b)$.

\item $F$ is at least $\epsilon_2$-far from a dictatorship and
from an anti-dictatorship (i.e., an SCF which always outputs the
bottom choice of a fixed voter).

\item $F$ is at least $\epsilon_2$-far from breaching non-imposition.
That is, for each alternative $a$, $\Pr_{x \in (L_3)^n} [F(x)=a] \geq
\epsilon_2$.
\end{itemize}
Then one can construct a GSWF $G$ on three alternatives, such
that:
\begin{enumerate}
\item $G$ satisfies the IIA condition.

\item $Dist(G,TR_3) \geq \epsilon_2 - 3 \sqrt{\epsilon_1}$.

\item $NT(G) \leq 3 \sqrt{\epsilon_1}$.
\end{enumerate}
\end{lemma}

\begin{proof}
Given $F$, we define the GSWF $G$ as follows:
\begin{definition}
For each pair of alternatives $a,b$, and a profile $x \in
(L_3)^n$, we set $G^{a,b}(x)=1$ if
\[
\Pr_{x'} [F(x')=a \ | \ x'^{a,b} = x^{a,b}] > \Pr_{x'} [F(x')=b \
| \ x'^{a,b} = x^{a,b}],
\]
and $G^{a,b}(x)=0$ if the reverse inequality holds. In the case of
equality we break the tie according to the preference of some
fixed voter between $a$ and $b$.
\end{definition}
Intuitively, $G^{a,b}(x)$ considers all profiles $x'$ which agree
with $x$ on the preferences of the voters between $a$ and $b$, and
checks whether $F(x')=a$ occurs more often then $F(x')=b$ or the
opposite, while ignoring all cases where $F(x')$ equals some other
alternative. It is clear from the definition that $G$ satisfies
the IIA condition, and that if $F$ is neutral (i.e., invariant
under permutation of the alternatives), then $G$ is neutral as
well.

In order to analyze $G$, we introduce an auxiliary definition:
\begin{definition}\label{maj}
A profile $x \in (L_3)^n$ is called a {\em minority preference on
the pair of alternatives $(a,b)$} if $F(x)=a$ while
$G^{a,b}(x)=0$, or if $F(x)=b$ while $F^{a,b}(x)=1$. $x$ is called
a {\em minority preference} if it is a minority preference for at
least some pair $(a,b)$. For a fixed pair of alternatives $a,b$,
denote
\[
N^{a,b}(F)=\Pr_{x \in (L_3)^n} [x \: \mbox{is a minority
preference on $(a,b)$}].
\]
\end{definition}

It is easy to relate $N^{a,b}$ to $M^{a,b}$, using the
Cauchy-Schwarz inequality:

\begin{proposition}\label{mstarn}
For every SCF $F$ and every pair of alternatives $a,b$ we have
\[
M^{a,b}(F) \ge (N^{a,b}(F))^2.
\]
\end{proposition}
\begin{proof}
Given $F,a,b$, and a vector $x^{a,b} \in \{0,1\}^n$ representing
the preferences of the voters between $a$ and $b$, define
\[
p_a(x^{a,b}) = \Pr_{x'} [F(x')=a \ | \ x'^{a,b} = x^{a,b}]
\]
and
\[
p_b(x^{a,b}) = \Pr_{x'} [F(x')=b \ | \ x'^{a,b} = x^{a,b}].
\]
In these terms,
\[
M^{a,b}(F)=\mathbb{E}_{x^{a,b} \in \{0,1\}^n} [p_a(x^{a,b}) \cdot
p_b(x^{a,b})],
\]
while
\[
N^{a,b}(F) = \mathbb{E}_{x^{a,b} \in \{0,1\}^n} [\min
\{p_a(x^{a,b}),p_b(x^{a,b})\}].
\]
Thus, by the Cauchy-Schwarz inequality,
\[
M^{a,b}(F)=\mathbb{E} [p_a \cdot p_b] \ge
\mathbb{E}[(\min\{p_a,p_b\})^2] \ge (\mathbb{E}
[\min\{p_a,p_b\}])^2=(N^{a,b}(F))^2,
\]
as asserted.
\end{proof}

We are now ready to prove that $G$ satisfies the desired
properties.

Consider a profile $x$ that is not a minority preference and
denote $a=F(x)$. Note that by the definition of a minority
preference, for all $b$ we must have that $G^{a,b}(x)=1$ and thus,
$a$ is a Generalized Condorcet Winner of $G$ at $x$.

This immediately implies that $G$ satisfies Assertion~3 of the lemma.
Indeed,
\begin{align*}
NT(G) &= \Pr_x[\mbox{ G does not have a GCW at x }]
\\
&\leq \Pr[\mbox{ x is a minority preference of F}] \\
&\leq \sum_{a,b} N^{a,b}(F) \leq \sum_{a,b} \sqrt{M^{a,b}(F)} \leq
3 \sqrt{\epsilon_1},
\end{align*}
as asserted.

In order to prove Assertion~2, let $Dist(G,TR_3)=\epsilon$, and
let $H \in TR_3$ be such that $G$ can be transformed to $H$ by
changing only fraction $\epsilon$ of the values. We consider four
cases:
\begin{enumerate}
\item \textbf{Case~1: $H$ always ranks alternative $a$ at the
top.} In this case, $\Pr[\mbox{ a is a GCW of G }] \geq
1-\epsilon$. Note that by the argument above, if $x$ is not a
minority preference and $a$ is a $GCW$ of $G$ at $x$ then
$F(x)=a$. Hence,
\[
\Pr[F(x)=a] \geq (1-\epsilon) - \Pr[\mbox{ x is a minority
preference of F }] \geq 1-\epsilon- 3 \sqrt{\epsilon_1}.
\]
However, by the assumption,
\[
\Pr[F(x)=a] \leq 1- \Pr[F(x)=b] \leq 1- \epsilon_2,
\]
and thus, $\epsilon \geq \epsilon_2 - 3\sqrt{\epsilon_1}$, as
asserted.

\item \textbf{Case~2: $H$ always ranks alternative $a$ at the
bottom.} In this case, $\Pr[\mbox{ a is a GCL of G }] \geq
1-\epsilon$. As in the previous case, if $x$ is not a minority
preference and $a$ is a $GCL$ of $G$ at $x$ then $F(x) \neq a$.
Thus,
\[
\Pr[F(x)=a] \leq \epsilon+3 \sqrt{\epsilon_1}.
\]
However, by assumption we have $\Pr[F(x)=a] \geq \epsilon_2$, and
thus $\epsilon \geq \epsilon_2 - 3 \sqrt{\epsilon_1}$, as
asserted.

\item \textbf{Case~3:$H$ is a dictatorship of voter $i$.} For a
profile $x$, denote the top alternative in the preference order of
voter $i$ by $x_i^{top}$. We have
\[
\Pr[ x_i^{top} \mbox{ is a GCW of G at x }] \geq 1-\epsilon.
\]
As in the previous cases, this implies that
\[
\Pr[F(x)=x_i^{top}] \geq 1-\epsilon - 3\sqrt{\epsilon_1}.
\]
However, since by assumption, $F$ is at least $\epsilon_2$-far
from a dictatorship of voter $i$, we have $\epsilon +
3\sqrt{\epsilon_1} \geq \epsilon_2$, and the assertion follows.

\item \textbf{Case~4:$H$ is an anti-dictatorship of voter $i$.} By
the same argument as in the previous case, if $x_i^{bot}$ is the
bottom alternative in the preference order of voter $i$, then
\[
\Pr[F(x)=x_i^{bot}] \geq 1-\epsilon - 3\sqrt{\epsilon_1}.
\]
However, since $F$ is also at least $\epsilon_2$-far from
anti-dictatorship of voter $i$, the assertion follows.
\end{enumerate}
This completes the proof of Condition~2 and of
Lemma~\ref{lemma:second-reduction}.
\end{proof}

\begin{rem}
We note that a certain converse of
Lemma~\ref{lemma:second-reduction} is true as well. If we have a
GSWF $G$ satisfying the IIA condition such that $\Pr[\mbox{ G has
a GCW }] \geq 1-\epsilon$, then we can define an SCF $F$ to be
equal to the GCW of $G(x)$ if such GCW exists, and to the top
choice of a fixed voter if the GCW does not exist. Since $G$
satisfies the IIA condition, the event $F(x)=a$ and $F(x')=b$ with
$x^{a,b}=x'^{a,b}$ can occur only if either $G(x)$ or $G(x')$ does
not have a GCW, and thus, $M^{a,b}(F) \le 2 \epsilon$.
\end{rem}

\section{Application of a Quantitative Arrow Theorem}
\label{sec:quantitative-arrow}

The only ingredient required for concluding the proof of
Theorem~\ref{t:mt} is a quantitative version of Arrow's
impossibility theorem. In order to get the optimal result for
various assumptions on the SCF $F$, we use two such versions, due
to Kalai~\cite{Ka02}, and to Keller~\cite{Kel10}.
\begin{theorem}\label{Thm:Quant-Arrow}
Let $G$ be a GSWF on three alternatives which satisfies the IIA
condition. Then:
\begin{enumerate}
\item If $Dist(G,TR_3) \geq \epsilon$, then $NT(G) \geq C_1 \cdot
\epsilon^3$, where $C_1$ is a universal constant.~\cite{Kel10}

\item If, in addition, $G$ is neutral and is at least
$\epsilon$-far from a dictatorship and an anti-dictatorship, then
$NT(G) \geq C_2 \cdot \epsilon$, where $C_2$ is a universal
constant.~\cite{Ka02}
\end{enumerate}
\end{theorem}

Now we are ready to present the proof of Theorem~\ref{t:mt}.

\begin{proof}
Let $F$ be an SCF on three alternatives, and assume on the
contrary that:
\begin{itemize}
\item The distance of $F$ from a dictatorship is at least
$\epsilon$,\footnote{We note that there is no need to add the
condition that $F$ is far from an anti-dictatorship, since an SCF
which is close to an anti-dictatorship can be clearly manipulated
by the ``anti-dictator''.}

\item For each alternative $a$, $\Pr[F(x)=a] \geq \epsilon$, but

\item $\sum_i M_i(F) < C_0 \cdot \epsilon^6$ (where $C_0$ is a
constant that will be specified below).
\end{itemize}

By Lemma~\ref{lemma:first-reduction}, it follows that for each
pair of alternatives $a,b$, we have $M^{a,b}(F) < 6C_0 \cdot
\epsilon^6$. By Lemma~\ref{lemma:second-reduction}, it then
follows that there exists a GSWF $G$ on three alternatives which
satisfies the IIA condition, and
\begin{itemize}
\item $Dist(G,TR_3) \geq \epsilon - 3 \sqrt{6C_0 \epsilon^6} \geq
\epsilon/2$. (The second inequality holds for $C_0$ sufficiently
small.)

\item $NT(G) < 3 \sqrt{6C_0 \epsilon^6} = 3\sqrt{6C_0} \cdot
\epsilon^3$.
\end{itemize}
However, for $C_0$ small enough (concretely, $C_0 \leq C_1^2/3456$
where $C_1$ is the constant in Theorem~\ref{Thm:Quant-Arrow}),
this contradicts the first version of
Theorem~\ref{Thm:Quant-Arrow} above. This proves the first
assertion of Theorem~\ref{t:mt}. The second assertion follows
similarly using the second version of
Theorem~\ref{Thm:Quant-Arrow} instead of the first one (note that
by the construction of $G$, if $F$ is neutral then $G$ is neutral
as well and thus Kalai's version of the quantitative Arrow theorem
can be applied). This completes the proof of Theorem~\ref{t:mt}.
\end{proof}

\begin{rem}\label{Rem:Const}
Since the value of the constant $C_1$ in the first version of
Theorem~\ref{Thm:Quant-Arrow} is extremely low (i.e., of order
$\exp(2^{-10,000,000})$ ), for certain values of $n$ and
$\epsilon$, a better result can be obtained by using another
version of the quantitative Arrow theorem. In that version,
obtained by Mossel~\cite{Mos09}, the lower bound $C \cdot
\epsilon^3$ is replaced by $(1/36000) \cdot \epsilon^3 n^{-3}$.
Applying Mossel's theorem instead of the version we used above, we
get the lower bound
\[
\sum_i M_i(F) \geq C' \cdot \epsilon^6 / n^6,
\]
where $C' \approx 10^{-8}$. While this bound depends also on $n$,
for certain values of the parameters it is still stronger, due to
the bigger value of the constant.
\end{rem}

\section{SCFs with More than Three Alternatives}
\label{sec:more-alternatives}

In this section we consider SCFs with more than three
alternatives. We show that the second step of our proof (reduction
from an SCF with low dependence on irrelevant alternatives to an
almost transitive GSWF) can be generalized to SCFs on $m$
alternatives, and that the third step (application of a
quantitative Arrow theorem) can be generalized under an additional
assumption of neutrality. However, we weren't able to generalize
the first step (reduction from low manipulation power to low
dependence on irrelevant alternatives), and thus we do not obtain
any variant of the main theorem for more than three alternatives.

We would like to mention again two related follow-up works. Xia
and Conitzer \cite{XC08} showed that the first step of our proof
can be generalized to any constant number of alternatives under
several additional assumptions (see
Section~\ref{sec:related-work}). Furthermore, Isaksson et
al.~\cite{IKM10} obtained a quantitative Gibbard-Satterthwaite
theorem for any number of alternatives under a single additional
assumption of neutrality, using a different technique.

Despite these two works, we decided to present the partial
generalization of our proof to more than three alternatives,
hoping that the technique can be extended to obtain a quantitative
Gibbard-Satterthwaite theorem without the neutrality assumption.

\subsection{Reduction from an SCF with Low Dependence on Irrelevant
Alternatives to a GSWF which Almost Always has a Condorcet Winner}

In order to present the results of this section, we have to
generalize the notions of $TR_3$ and $NT(G)$ defined in
Section~\ref{sec:second-reduction} to GSWFs on $m$ alternatives.

The class $TR_m$ of all GSWFs on $m$ alternatives which satisfy
the IIA condition and whose output is always transitive, was
partially characterized by Wilson~\cite{Wil72}, and fully
characterized by Mossel~\cite{Mos09} in the following theorem:
\begin{theorem}[Mossel]\label{Thm:Mossel3}
The class $TR_m$ consists exactly of all GSWFs $G$ on $m$
alternatives satisfying the IIA condition, for which there exists
a partition of the set of alternatives into disjoint sets
$A_1,A_2,\ldots,A_r$ such that:
\begin{itemize}
\item For any profile, $G$ ranks all the alternatives in $A_i$
above all the alternatives in $A_j$, for all $i<j$.

\item For all $s$ such that $|A_s| \geq 3$, the restriction of $G$
to the alternatives in $A_s$ is a dictatorship or an
anti-dictatorship.

\item For all $s$ such that $|A_s|=2$, the restriction of $G$ to
the alternatives in $A_s$ is an arbitrary non-constant function of
the individual preferences between the two alternatives in $A_s$.
\end{itemize}
\end{theorem}

While the notion $NT(G)$ makes sense also for GSWFs on $m$
alternatives, we use here a different notion which coincides with
$NT(G)$ in the case of three alternatives:
\begin{notation}
Let $G$ be a GSWF on $m$ alternatives. The probability that $G$
does not have a Generalized Condorcet Winner (GCW) is denoted by
\[
NGCW(G) = \Pr_{x \in (L_m)^n} [\mbox{ G does not have a GCW at x
}].
\]
Similarly, $GCW(G)$ denotes the probability that $G$ has a GCW.
\end{notation}

Under these definitions, Lemma~\ref{lemma:second-reduction}
generalizes directly to the case of $m$ alternatives. We get:
\begin{lemma}\label{lemma:second-reduction-general}
Let $\epsilon_1,\epsilon_2>0$, and let $F$ be an SCF on $m$
alternatives, such that:
\begin{itemize}
\item $M^{a,b}(F) \leq \epsilon_1$ for all pairs $(a,b)$.

\item $F$ is at least $\epsilon_2$-far from a dictatorship and
from an anti-dictatorship.

\item There exist alternatives $a,b,c$, such that
\[
\min \left( \Pr_{x \in (L_m)^n} [F(x)=a], \Pr_{x \in (L_m)^n}
[F(x)=b], \Pr_{x \in (L_m)^n} [F(x)=c] \right) \geq \epsilon_2.
\]
\end{itemize}
Then one can construct a GSWF $G$ on $m$ alternatives, such that:
\begin{enumerate}
\item $G$ satisfies the IIA condition.

\item $Dist(G,TR_m) \geq \epsilon_2 - {{m}\choose{2}} \cdot
\sqrt{\epsilon_1}$.

\item $NGCW(G) \leq {{m}\choose{2}} \cdot \sqrt{\epsilon_1}$.
\end{enumerate}
Furthermore, if $F$ is neutral, then $G$ is neutral as well.
\end{lemma}

Note that the third condition imposed on $F$, which means that $F$
is $\epsilon_2$-far from having only two alternatives in its
range, is weaker than being $\epsilon_2$-far from breaching
non-imposition (which means that any alternative is elected with
probability at least $\epsilon_2$).

\medskip

The proof of Lemma~\ref{lemma:second-reduction-general} is very
similar to the proof of Lemma~\ref{lemma:second-reduction}, and
thus we present only the required modifications.

\medskip

\noindent \textit{Sketch of Proof.} The definition of $G$ and the
proofs of Assertions~1 and~3 are the same as in the proof of
Lemma~\ref{lemma:second-reduction}. In order to prove Assertion~2,
let $Dist(G,TR_m)=\epsilon$, and let $H \in TR_m$ be such that $G$
can be transformed to $H$ by changing only fraction $\epsilon$ of
the values. By Theorem~\ref{Thm:Mossel3} applied to $H$, the set
of alternatives can be partitioned into disjoint sets
$A_1,A_2,\ldots,A_r$ such that for any profile, $H$ ranks all the
alternatives in $A_i$ above all the alternatives in $A_j$, for all
$i<j$. Note that for any alternative $d \not \in A_1$, $F(x)=d$
can occur only if either $G(x) \neq H(x)$ or $x$ is a minority
preference. Hence,
\[
\Pr_x [F(x)=d] \leq \epsilon + {{m}\choose{2}} \cdot
\sqrt{\epsilon_1}.
\]
Therefore, either $\epsilon_2 \leq \epsilon + {{m}\choose{2}}
\cdot \sqrt{\epsilon_1}$ (as claimed in Assertion~2 of the lemma),
or $a,b,c \in A_1$. In the latter case, $|A_1| \geq 3$ and thus,
by Theorem~\ref{Thm:Mossel3}, the restriction of $H$ to the
alternatives in $A_1$ is a dictatorship or an anti-dictatorship.
In both cases, the assertion of the lemma is proved in the same
way as cases~3 and~4 in the proof of
Lemma~\ref{lemma:second-reduction}. $\Box$

\subsection{Generalization of the Quantitative Arrow Theorem}

The quantitative versions of Arrow's theorem presented by Mossel~\cite{Mos09}
and Keller~\cite{Kel10} apply also to GSWFs on $m$ alternatives, and
assert that if $Dist(G,TR_m)$ is not too small, then $NT(G)$ is
also not too small. However, the reduction given by
Lemma~\ref{lemma:second-reduction-general} yields a bound on
$NGCW(G)$ rather than on $NT(G)$, and thus we need a lower bound
on $NGCW(G)$ (the probability of not having a Generalized
Condorcet Winner), which may be much lower than the probability of
being non-transitive.

In this subsection we prove a generalization of the quantitative
Arrow theorem which allows to obtain a lower bound on $NGCW(G)$.
However, we require the additional assumption that $G$ is neutral
(i.e., invariant under permutation of the alternatives), and our
proof relies heavily on this assumption.

Before we present the generalization, we recall a few properties
of neutral GSWFs. Let $G$ be a GSWF on $n$ voters and $m$
alternatives denoted by $\{1,2,\ldots,m\}$. If $G$ satisfies the
IIA condition, then its output is determined by ${{m}\choose{2}}$
Boolean functions $G^{i,j}:\{0,1\}^n \rightarrow \{0,1\}$, which
are given the individual preferences between alternatives $i$ and
$j$ and output the preference of the society between them. If, in
addition, $G$ is neutral, then all the functions $G^{i,j}$ are
equal, and thus we denote them by a single function $g:\{0,1\}^n
\rightarrow \{0,1\}$, and write $G=g^{\otimes {{m}\choose{2}}}$.
Note that the neutrality assumption implies also that $g$ is an
odd function, that is,
$g(a_1,\ldots,a_n)=1-g(1-a_1,\ldots,1-a_n)$, and in particular,
$\Pr[g=1]=1/2$. We denote the distance of $g$ from a dictatorship
or an anti-dictatorship by $Dist(g,DICT_2)$.

We are now ready to present our result. We start with an
equivalent formulation of Kalai's version of the quantitative
Arrow theorem~\cite{Ka02}.
\begin{theorem}[Kalai]
There exists a constant $C_3 >0 $ such that the following holds.
Let $G=g^{\otimes {{3}\choose{2}}}$ be a neutral GSWF on $n$
voters and $3$ alternatives satisfying the IIA condition. Then
\[
NGCW(G) \geq C_3 \cdot Dist(g,DICT_2).
\]
\end{theorem}

We prove the following generalization:
\begin{theorem}\label{Thm:NGCW}
For any $\epsilon>0$, and for every $m \geq 3$, there exists a
constant $\delta_m(\epsilon) > 0 $ such that the following holds.
Let $G=g^{\otimes {{m}\choose{2}}}$ be a neutral GSWF on $n$
voters and $m$ alternatives which satisfies the IIA condition. If
$Dist(g,DICT_2) \geq \epsilon$, then $NGCW(G) \geq \delta_m$.

Moreover, for $m=3,4,5$, we can take $\delta_m = C \cdot
\epsilon$, where $C$ is a universal constant.
\end{theorem}

Before we present the proof of the theorem, we note that if a
neutral GSWF $G=g^{\otimes {{m}\choose{2}}}$ on $m$ alternatives
is at least $\epsilon$-far from a dictatorship and from an
anti-dictatorship, then $Dist(g,Dict_2) \geq
\epsilon/{{m}\choose{2}}$. Thus, Theorem~\ref{Thm:NGCW} implies
immediately the following corollary.
\begin{corollary}
For any $\epsilon>0$, and for every $m \geq 3$, there exists a
constant $\delta'_m(\epsilon) > 0 $ such that the following holds.
Let $G$ be a neutral GSWF on $m$ alternatives satisfying the IIA
condition. If $G$ is at least $\epsilon$-far from a dictatorship
and from an anti-dictatorship, then $NGCW(G) \geq \delta'_m$.

Moreover, for $m=3,4,5$, we can take $\delta'_m = C \cdot
\epsilon$, where $C$ is a universal constant.
\end{corollary}

\noindent \textbf{Proof of Theorem~\ref{Thm:NGCW}.}
The case $m=3$ is exactly Kalai's theorem above. We first give a
direct proof of the cases $m=4$ and $m=5$, and then show a general
inductive argument that allows to leverage the result to any
$m>5$.

\bigskip

\noindent \textbf{GSWFs on four alternatives}

\bigskip

We begin by considering the case $m=4$. For $1 \leq i,j \leq 4$,
let $X_{ij}$ be the random $0/1$ variable that indicates the event
$G^{i,j}(x)=1$, where the profile $x$ is chosen at random. Note
that by the neutrality assumption, the probability $GCW(G)$ is
precisely four times the probability that alternative $1$ is a GCW
of $G$. Hence,
\begin{equation}\label{inex4}
GCW(G) = 4 \cdot \mathbb{E}[\prod_{j=2}^4 X_{1j}] = 4 \cdot
\mathbb{E}[\prod_{j=2}^4 (1-X_{j1})].
\end{equation}
Before expanding this equation, we make three observations. First,
from the neutrality of $G$ it follows that $g$ is balanced (i.e.,
$\Pr[g=1]=1/2$), and thus, for all $j \in \{2,3,4\}$ we have
\[
\mathbb{E}[X_{j1}]=1/2.
\]
Next, for any pair $i,j$ with $2 \leq i,j \leq 4$, we can apply Kalai's
theorem to the GSWF $G'$ which is the restriction of $G$ to the
alternatives $\{1,i,j\}$ to get:
\begin{align*}
\mathbb{E}[X_{j1}X_{i1}] =& \Pr[\mbox{ 1 is a GCL of G' }] \\
=& \frac{1}{3} \cdot GCW(G') \leq \frac 1 3 (1- C_3 \cdot
Dist(g,DICT_2)).
\end{align*}
Finally, from neutrality, the probability that alternative $1$ is
a GCW is precisely equal to the probability that he is a GCL, and
thus,
\[
\mathbb{E}[\prod_{j=2}^4 X_{1j}] = \mathbb{E}[\prod_{j=2}^4
X_{j1}].
\]
Using these observations we expand Equation~(\ref{inex4}) and get:
\[
GCW(G) = 4 ( 1 -3\cdot \frac 1 2 + 3 \cdot \frac{1}{3} \cdot
GCW(G') - \frac{1}{4} \cdot GCW(G)),
\]
and thus, by Kalai's theorem,
\begin{equation} \label{34}
GCW(G) = 2GCW(G') - 1  \leq 1-2C_3 \cdot Dist(g,DICT_2),
\end{equation}
which yields the assertion of the theorem for $m=4$ with $\delta_4
= 2C_3 \cdot \epsilon$.

\bigskip

\noindent \textbf{GSWFs on five alternatives}

\bigskip

Next we consider the case $m=5$. Unfortunately, the first natural
step, generalizing the inclusion-exclusion type
formula~(\ref{inex4}) to get
\[
GCW(G)= 5 \cdot \mathbb{E}[\prod_{j=2}^5 (1-X_{j1})],
\]
does not help, due to an annoying prosaic reason: the two terms $
\mathbb{E}[\prod_{j=2}^5 (X_{1j})]$ and $\mathbb{E}[\prod_{j=2}^5
(1-X_{j1})] $ which appear on the two sides of the equation have
the same sign, and cancel out. To remedy this we consider a
neutral GSWF $G_6$ on six alternatives, and denote its
restrictions to five, four, and three alternatives by $G_5,G_4,$
and $G_3$, respectively. We start with the expansion
\[
GCW(G_6) = 6 \cdot \mathbb{E}[\prod_{j=2}^6 (1-X_{j1})],
\]
which gives:
\begin{align*}
GCW(G_6) =& 6 ( 1 - \frac 5 2 + {5 \choose 2}
\frac{GCW(G_3)}{3}- \\
&- {5 \choose 3} \frac{GCW(G_4)}{4} +{5 \choose 4 }
\frac{GCW(G_5)}{5} - \frac{GCW(G_6)}{6} ).
\end{align*}
Rearranging, and using Equation~(\ref{34}), we get:
\[
\frac{GCW(G_6)}{3} + \frac{5 \cdot GCW(G_3)}{3} - 1 = GCW(G_5).
\]
Since $GCW(G_6) \leq 1$ and $ GCW(G_3) \leq 1- C_3 \cdot
Dist(g,DICT_2)$, this yields
\[
GCW(G_5) \leq 1 - \frac 5 3 C_3 \cdot Dist(g,DICT_2),
\]
which is the assertion of the theorem for $m=5$ with
$\delta_5=\frac 5 3 C_3 \cdot \epsilon$.

\bigskip

\noindent \textbf{GSWFs on more than five alternatives}

\bigskip

The assertion of the theorem for all $m>5$ follows from the cases
$m=3,4,5$ using full induction.

Assume that we already proved the assertion for $m_1$ and $m_2$,
and let $G=g^{\otimes {{m_1+m_2}\choose{2}}}$ be a GSWF on
$m_1+m_2$ alternatives, such that $Dist(g,Dict_2)=\epsilon$.
Applying the assertion to the restrictions of $G$ to the first
$m_1$ alternatives and to the last $m_2$ alternatives, we get that
with probability at least $\delta_{m_1}$, there is no GCW among
the first $m_1$ alternatives, and with probability at least
$\delta_{m_2}$ there is no GCW among the last $m_2$ alternatives.
The key point here is that these two events are independent since
the voter preferences within two disjoint sets of alternatives are
totally independent of each other. Thus, the probability that
there is no GCW at all is at least $\delta_{m_1} \cdot
\delta_{m_2}$.

Starting with $\delta_m = C \cdot \epsilon$ for $m=3,4,5$, we get
$\delta_m \ge (C \epsilon)^{\lfloor m/3 \rfloor}$ for general $m$.
(Luckily, every integer $m>5$ can be represented as $m=3a+4b$,
with $a$ and $b$ nonnegative integers.)

This completes the proof of Theorem~\ref{Thm:NGCW}. $\Box$

\begin{rem}
We note that the value of $\delta_m$ obtained in our proof for
general values of $m$ decreases as $\epsilon^{O(m)}$. However, we
conjecture that for a fixed $\epsilon>0$, not only that $\delta_m$
need not decrease with $m$, it actually tends to 1. This
conjecture is supported by a recent work of Mossel~\cite{Mossel07}
who calculated the asymptotic value $\lim_m \lim_n [1- \delta(m)]=
\Theta(1/m)$ for the case when $G^{a,b}: \{0,1\}^n \ra \{0,1\}$ is
a low-influence function (e.g., the majority function) on
$x^{a,b}$ for all $a,b \in \{1,2,\ldots,m\}$.
\end{rem}

\end{document}